\documentclass[12pt]{article}
\usepackage{epsfig,amsmath,amssymb,graphicx}
\usepackage{comment}

\newtheorem{lemma}{Lemma}[section]
\newtheorem{theorem}[lemma]{Theorem}
\newtheorem{proposition}[lemma]{Proposition}
\newtheorem{corollary}[lemma]{Corollary}
\newtheorem{conjecture}[lemma]{Conjecture}

\newcommand {\qed}{\rule{2mm}{2mm}\medskip}
\newenvironment {proof}{\noindent{\bf Proof:}~}{~\qed}

\begin{document}

% THE TITLE PAGE 

\begin{center} 
{\bf \LARGE Heffter Arrays and \\
\vspace*{1.0ex}
Biembedding Graphs on Surfaces}
\end{center}

\vspace*{2.0ex}

\begin{center}
   Dan Archdeacon \\
   Dept. of Math. and Stat. \\
   University of Vermont \\
   Burlington, VT 05405 \ \ USA\\
   {\tt dan.archdeacon@uvm.edu}
\end{center}

%\begin{center}
 %\begin{minipage}[t]{.45\textwidth}
 %   \begin{center}
 %     Dan Archdeacon \\
 %     Dept. of Math. and Stat. \\
 %     University of Vermont \\
 %     Burlington, VT 05405 \ \ USA\\
 %     {\tt dan.archdeacon@uvm.edu}
 %   \end{center}
 % \end{minipage}\hspace*{.1\textwidth}
 % \begin{minipage}[t]{.45\textwidth}
 %   \begin{center}
 %     Jeff Dinitz \\
 %     Dept. of Math. and Stat. \\
 %     University of Vermont \\
 %     Burlington, VT 05405 \ \ USA\\
 %     {\tt jeff.dinitz@uvm.edu}
 %   \end{center}
 % \end{minipage}
%\end{center}
%\vspace*{1.0ex}

\vspace*{1.0ex}
\begin{center}
%{\bf Draft: Not for Distribution}\\
%\vspace*{1.0ex}
\today
\end{center}

%\vspace*{1.0ex}

\begin{abstract} A Heffter array is an $m \times n$ matrix with nonzero entries from $\mathbb{Z}_{2mn+1}$ such that {\em i)} every row and column sum to 0, and {\em ii)} no element from $\{x,-x\}$ appears twice. We construct some Heffter arrays. These arrays are used to build current graphs used in topological graph theory. In turn, the current graphs are used to embed the complete graph $K_{2mn+1}$ so that the faces can be 2-colored, called a biembedding. Under certain conditions each color class forms a cycle system. These generalize biembeddings of Steiner triple systems. We discuss some variations including Heffter arrays with empty cells, embeddings on nonorientable surfaces, complete multigraphs, and using integer in place of modular arithmetic.
\end{abstract}

% END OF TITLE PAGE

%%%%%%%%%%%%%%%%%%%%%%%%%%%%%%%%%%%%%%%%%%%%%%
\section {Introduction}\label{introduction}
%%%%%%%%%%%%%%%%%%%%%%%%%%%%%%%%%%%%%%%%%%%%%%

We study a relation between design theory, graph theory, and maps on surfaces. From design theory a Heffter system is used to construct a cyclic $k$-cycle system. We introduce orthogonal Heffter systems and represent them as a Heffter array. The array is related to a current assignment on the complete bipartite graph $K_{m,n}$. The current graph with certain conditions is then used to construct an orientable embedding of the complete graph $K_{2mn+1}$ that is face 2-colorable, the boundaries of each color class forming a cycle system. A similar theorem is given for embeddings in nonorientable surfaces. 

{\em Heffter's First Difference Problem} \cite{H} asks if the numbers from 1 to $(m-1)/2$ can be partitioned into $(m-1)/6$ triples $(x,y,z)$ such that either $x+y=z$ or $x+y+z = m$. Heffter used this partition to construct a {\em Steiner triple system}, {\em STS($m$)},  a collection of triples from an $m$-set that collectively contain every pair exactly once \cite{CR}. This corresponds to a set of 3-cycles whose edges partition $E(K_m)$. An {\em $s$-cycle system} partitions $E(K_m)$ into $s$-cycles. Buratti and Del Fra \cite{BD} proved the existence of $k$-cycle systems of $K_m$ having a cyclic action on the parts whenever $m \equiv 1 \pmod{2k}$. 

A Heffter array is an $m \times n$ array with non-zero entries from $\mathbb{Z}_{2mn+1}$ such that the entries are all distinct up to sign and such that each row and column sum to 0. We explore Heffter arrays in Section \ref{heffter}. We also define two properties of orderings $\omega_r$ and $\omega_c$ on the cells of a Heffter array. 

A current graph is an embedded graph where each directed edge has been assigned an element from a fixed current group. Under some special conditions current graphs can be used to construct embeddings of complete graph. These embeddings were first used in the solution of the Map Color Theorem \cite{R}. Details of this relation are given in \cite{GT}; we give a brief explanation in Section \ref{current_graphs}. The relation between Heffter arrays and current graphs are described in Section \ref{relation}. One main result is the following whose proof is given in Section \ref{relation}. 

\begin{theorem}\label{main} Given a Heffter array $H(m,n;s,t)$ with compatible orderings $\omega_r$ on $D(m,s)$ and $\omega_c$ on $D(n,t)$, there exists an orientable embedding of $K_{2ms+1}$ such that every edge is on a face of size $s$ and a face of size $t$. Moreover, if $\omega_r$ and $\omega_c$ are both simple, then all faces are simple cycles. \end{theorem}

A {\em biembedding} of the complete graph $K_m$ is one that is face 2-colorable. We are particularly interested when the face boundaries of the first color class form an $s$-cycle system and those of the other color class form a $t$-cycle system. These have most commonly been studied in the case $s = t = 3$, that is, both color classes are {\em STS$(m)$}'s. Triangular biembeddings were shown to exist in \cite{R} for all $n \equiv 5 \pmod{12}$ and in \cite{Y} for all $n \equiv 7 \pmod{12}$, the two necessary cases. Biembeddings of Steiner triple systems have been widely studied \cite{GK1,GK2} particularly for small values \cite{GGK}. Grannell and Griggs \cite{GG} give a very nice survey. McCourt \cite{M} has studied biembeddings where one color class gives and {\em STS} and the other half give a decomposition of $E(K_m)$ into Hamilton cycles. Brown \cite{B} has a class of embeddings where one color class is triangles and the other quadrilaterals.

Section \ref{nonorientable} discusses weak Heffter arrays and their use to construct biembeddings on nonorientable surfaces. Section \ref{conclusion} closes with some directions for future research, several of which will be the subject of subsequent papers. 

%%%%%%%%%%%%%%%%%%%%%%%%%%%%%%%%%%%%%%%%%%%%%%
\section {Heffter systems and Heffter arrays}\label{heffter}
%%%%%%%%%%%%%%%%%%%%%%%%%%%%%%%%%%%%%%%%%%%%%%

Let $\mathbb Z_m$ be the cyclic group of odd order $m$ whose elements are denoted 0 and $\pm i$ where $i = 1,\dots,(m-1)/2$. A {\em half-set} $L$ is a subset of $(m-1)/2$ nonzero elements that contains exactly one of each pair $\{x,-x\}$. A {\em Heffter system}, $D(m,k)$, is a partition of $L$ into parts of size $k$ such that the elements in each part sums to 0 modulo $m$. {\em Heffter's First Difference Problem} \cite{H} asks if the numbers from 1 to $(m-1)/2$ can be partitioned into $(m-1)/6$ triples $(x,y,z)$ such that either $x+y=z$ or $x+y+z = m$. This is equivalent to finding a $D(m,3)$. 

Two Heffter systems $D_m = D(2mn+1,m)$ and $D_n = D(2mn+1,n)$ on the same half-set $L$ are {\em orthogonal} if each $m$-set of $D_m$ intersects each $n$-set of $D_n$ in a single element. A {\em Heffter array} $H(m,n)$ is an $m \times n$ array whose rows form a $D(2mn+1,n)$ and whose columns form a $D(2mn+1,m)$; we call these the {\em row} and {\em column Heffter systems} respectively. A Heffter array $H(m,n)$ is equivalent to a pair of orthogonal Heffter systems: cell $a_{i,j}$ contains the common element in the $i^{th}$ part of the row system and the $j^{th}$ part of the column system. Figure \ref{three_by_four} shows a Heffter array $H(3,4)$. 

\begin{figure}[htb]
\begin{center} 
$$ 
\begin{array}{|c|c|c|c|} \hline
1 & -2 & -10 & 11 \\ \hline
-8 & 6 & -3 & 5 \\ \hline
7 & -4 & -12 & 9 \\ \hline
\end{array}
$$\end{center} 
\caption{\label{three_by_four} A Heffter array $H(3,4)$ over $\mathbb Z_{25}$}
\end{figure}

Let $A$ be a subset of $\mathbb Z_{m}$ with $\sum_{a \in A} a \equiv 0 \pmod{m}$ such that no pair $\{x,-x\}$ is a subset of $A$. Consider a cyclic ordering $(a_1,\dots,a_k)$ of the elements in $A$ and let $s_i = \sum_{j=1}^i a_j$. The ordering is {\em simple} if $s_i \neq s_j$ for $i \neq j$. Equivalently, the cyclic ordering is simple if there is no consecutive subsequence of elements that sum to 0. A Heffter system $D(m,k)$ has a {\em simple ordering} if and only if each part has a simple ordering.

\begin{proposition}\label{cycle_decomposition} An ordered Heffter system $D(m,k)$ forms a decomposition of $E(K_m)$ into closed trails of length $k$. Moreover, if the ordering is simple, then its corresponding decomposition is a $k$-cycle system. 
\end{proposition} 

\begin{proof} Let $\{a_1,\dots,a_k\}$ be a part of $D(m,k)$ under any ordering. Form a closed trail $(0,s_1,s_2,\dots,s_k)$ in a complete graph $K_m$ with vertex set $\mathbb Z_m$. Develop this trail modulo $m$ and do the same for all other parts of $D(m,k)$. Since each pair $\{x,-x\}$ has exactly one element in $D(m,k)$, each difference appears once. Hence these closed trails partition $E(K_m)$. If the ordering on $D(m,k)$ is simple, then each of these trails are simple cycles. \end{proof} 

Let $D_1 = D(2mn+1,s)$, $D_2 = D(2mn+1,t)$ be two orthogonal Heffter systems with orderings $\omega_1$, $\omega_2$ respectively. The orderings are {\em compatible} if their composition $\omega_1 \circ \omega_2$ is a cyclic permutation on the half-set. The importance of compatible orderings will be examined in Section \ref{relation} when we relate Heffter arrays and current graphs. 

A variation of Heffter arrays allows for some cells to be empty. Two Heffter systems $D(2ms+1,s)$ and $D(2nt+1,t)$ on the same half-set of order $ms=nt$ are {\em sub-orthogonal} if each $s$-set of $D(2ms+1,s)$ intersects each $t$-set of $D(2nt+1,t)$ in at most one element. As before, form an $m \times n$ array $H(m,n;s,t)$ where $a_{i,j}$ is the common element in the $i^{th}$ part of $D(2ms+1,s)$ and the $j^{th}$ part of $D(2nt+1,t)$, if any, and the cell is empty otherwise. Necessary conditions for the existence of an $H(m,n;s,t)$ are $ms = nt$, $3 \le s \le m$, and $3 \le t \le n$. Figure \ref{unbalanced} gives an example of a $H(6,12;8,4)$.

\begin{figure}[htb]
\begin{center} 
$$ 
\begin{array}{|c|c|c|c|c|c|c|c|c|c|c|c|}
\hline
&& -1 & 2 &   5 & -6 &&& -25 & 26 & 29 & -30 \\ \hline
&& 3 & -4 & -7 &   8 &&& 27 & -28 & -31 & 32 \\ \hline
9 & -10 &&& -13 & 14 & 33 & -34 &&& -37 & 38 \\ \hline
-11 & 12&&& 15 & -16 & -35 & 36 &&& 39 & -40 \\ \hline
-17 & 18 & 21 & -22 &&& -41 & 42 & 45 & -46 && \\ \hline
 19 & -20 & -23 & 24 &&& 43 &-44 & -47 & 48 && \\ \hline
\end{array}
$$
\end{center} 
\caption{\label{unbalanced} A Heffter array $H(6,12;8,4)$}
\end{figure}

If a Heffter array $H(m,n;s,t)$ is square, i.e., $m=n$, then necessarily $s = t$. In this case we denote the square by $H(n;k)$. The commonality in the notation is that parameters before the 
semicolon refer to sides of the squares, those after to the number of filled cells in a row or column. Figure \ref{square_with_empty} gives a $H(5;4)$. Square arrays with empty cells are studied in \cite{ADDY}. 

\begin{figure}[htb]
\begin{center} 
$$ 
\begin{array}{|c|c|c|c|c|} \hline
 & 17 & -8 & -14&5  \\ \hline
1& & 18& -9 &-10\\ \hline
-6 & 2 &  &19 &-15\\ \hline
-11&-12&3&&20 \\ \hline
16&-7&-13&4& \\ \hline
\end{array}
$$\end{center} 
\caption{\label{square_with_empty} A Heffter array $H(5;4)$}
\end{figure}

Let $H(m_1+m_2,n_1+n_2;s,t)$ be a Heffter array. Suppose that the rows and columns of $H$ can be permuted such that each nonempty cell $a_{i,j}$ has either $i \le m_1$ and $j \le n_1$, or it has $i > m_1$ and $j > n_1$. Then the array is called {\em block diagonal}. Constructing block diagonal arrays is convenient and powerful, but they are not suited for the application to graph embeddings and so are sometimes avoided. 

%%%%%%%%%%%%%%%%%%%%%%%%%%%%%%%%%%%%%%%%%%%%%%
\section {Orientable embeddings and current graphs}\label{current_graphs}
%%%%%%%%%%%%%%%%%%%%%%%%%%%%%%%%%%%%%%%%%%%%%%

In this section we describe the use of rotations to describe a cellular embedding of a graph on a fixed orientable surface. We also define current graphs and their usefulness in embedding complete graphs. 

%===============================
\subsection{Orientable embeddings}
%===============================

Consider a graph $G$ and for every edge $e$ let $e^+$ and $e^-$ denote its two possible directions. Let $D(G)$ be the set of all directed edges, so $|D(G)| = 2 |E(G)|$, and define $\tau$ as the involution swapping $e^+$ and $e^-$ for every $e$. Let $D_v = \{ (v,u) \in D(G)\}$ denoted the edges directed out of $v$. A {\em local rotation} $\rho_v$ is a cyclic permutation of $D_v$. If we select a local rotation for each vertex, then collectively they form a {\em rotation} $\rho$ of $D(G)$. The orbits of $\rho$ correspond bijectively to the set of local rotations on $V(G)$. The proof of the following is omitted; see \cite{GT,MT} for details.

\begin{theorem} A rotation on $G$ is equivalent to a cellular embedding of $G$ in an oriented surface. The faces boundaries of the embedding corresponding to $\rho$ are the orbits of $\rho \circ \tau$. \end{theorem}

Calculating $\rho \circ \tau$ is called the {\em face-tracing algorithm}. Knowing the number of faces allows you calculate the genus $g$ of the surface using Euler's formula $|V| - |E| + |F| = 2 - 2 g$. 
 
We are especially interested in {\em monofacial} embeddings, those with a single face. By Euler's formula if a graph has a monofacial embedding, then $|V| \not \equiv |E| \pmod 2$. However, this necessary condition is not sufficient. The following is a special case of Xuong's Theorem \cite{X}. The proof provides an algorithm for calculating a rotation yielding the monofacial embedding. 

\begin{theorem}\label{xuong} A graph $G$ has a monofacial embedding on an orientable surface if and only if there is a spanning tree $T$ such that every component of $G - T$ has an even number of edges. \end{theorem}

Kundu \cite{K} showed that every 4-edge-connected graph has two disjoint spanning trees. This combined with Theorem \ref{xuong} gives:

\begin{corollary} If $G$ is 4-edge-connected and $|V| \not \equiv |E| \pmod 2$, then $G$ has a monofacial embedding. \end{corollary}

%===============================
\subsection{Current graphs}
%===============================

Current graphs \cite{GA} were originally developed as quotients of surface embeddings. In particular, a monofacial embedding of a graph with currents added to the edges are used to construct a rotation on a derived graph, usually complete. We briefly describe this construction.

A {\em current assignment} on $G$ with {\em currents} from $\mathbb Z_m$ is a function $\kappa : D(G) \rightarrow \mathbb Z_m$ such that $\kappa(e^-) = -\kappa(e^+)$. We frequently require the following conditions:
\begin{enumerate}
\item (Kirchoff's Current Law - {\em KCL}) For every vertex $v$, \ $\sum_{e \in D_v} \kappa (e) \equiv 0 \pmod m$,
\item (Unique Currents) $\kappa$ is a bijection between $D(G)$ and $\mathbb Z_m \setminus \{0\}$,
\item (Monofacial) $G$ is embedded on a surface with a single face (this property of $G$ is independent of $\kappa$).
\end{enumerate}

\begin{figure}[htb]
\begin{center} 
\includegraphics[width=.5\textwidth]{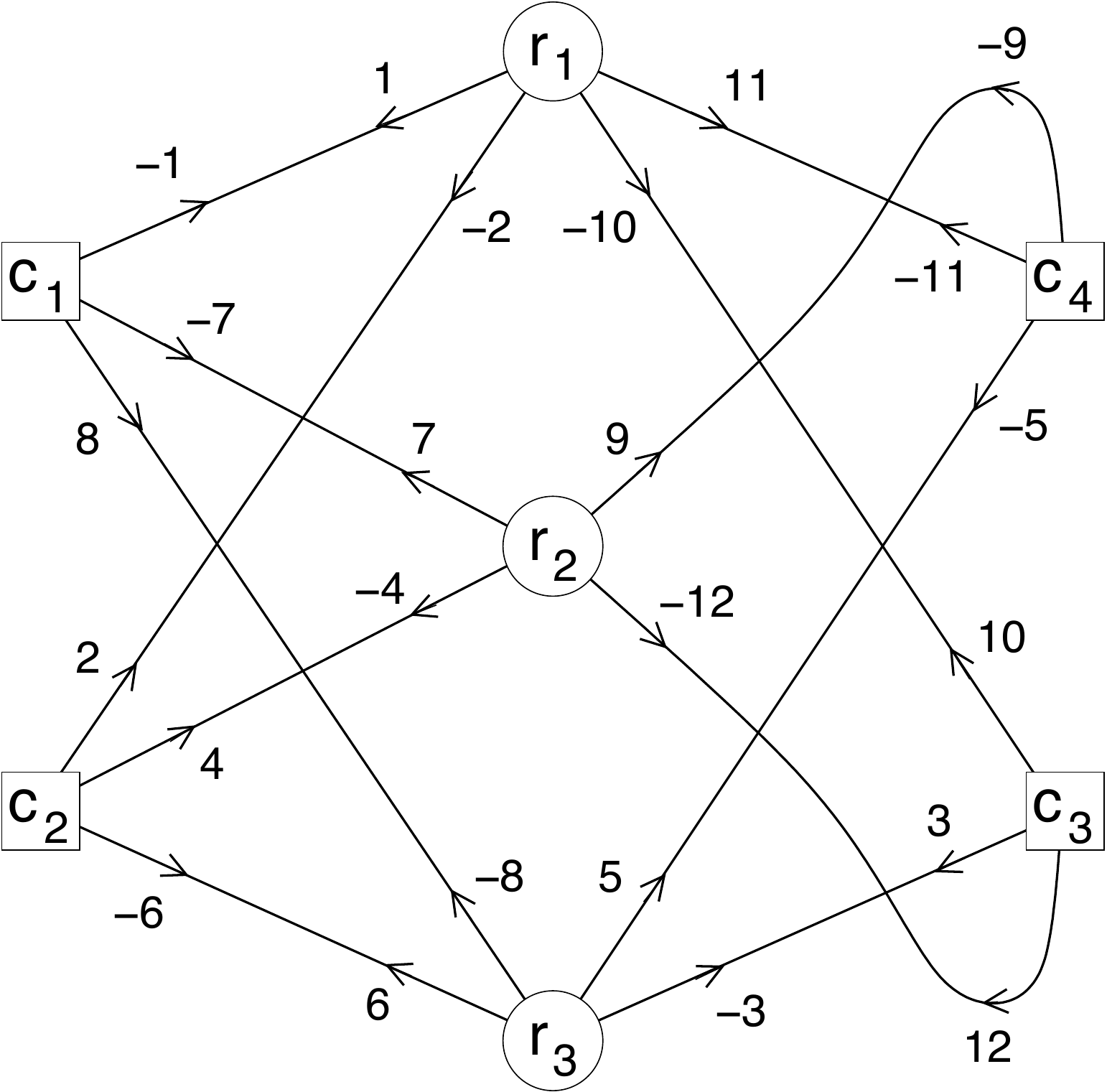} 
\end{center} 
\caption{\label{unsigned_k34} A current graph}
\end{figure}

Figure \ref{unsigned_k34} shows a unique-current assignment that satisfies {\em KCL} in $\mathbb Z_{25}$. Here, as elsewhere, we denote a directed edge $\bar e$ by its unique current $\kappa (\bar e)$. We read each local rotation anticlockwise as the edges emanate from the vertices; for example, around vertex $C_4$ the directed edges with currents $(-9,-11,-5)$ appear in that cyclic order. 
This gives the rotation:
\begin{eqnarray}
\label{rho} \rho = & (1,-2,-10,11)(7,-4,-13,9)(-8,6,-3,5)\\ 
\nonumber & (-1,8,-7)(2,-6,4)(10,3,12) (-11,-5,-9).
\end{eqnarray}
The face tracing algorithm gives a single face:
\begin{eqnarray}
\label{rhotau} \rho \circ \tau = & (1,8,6,4,-12,10,11,-5,-8,-7,-4,2, \\ 
\nonumber & -10,3,5,-9,7,-1,-2,-6,-3,12,9,-11).
\end{eqnarray}

A rotation $\rho$ on a current graph induces a local rotation on $D_v$ for each vertex $v$, say $(e_1,\dots,e_k)$. This local rotation is {\em simple with respect to $\kappa$} if the corresponding partial sums $s_i = \sum_{j=1}^i \kappa(e_j)$ are all distinct, similar to the definition of a simple ordering for a subset of a group.

An $(s,t)$-{\em biregular graph} with {\em biorder} $(m,n)$ is a bipartite graph with one part having $m$ vertices of degree $s$ and the other part having $n$ vertices of degree $t$.

\begin{theorem}\label{big_theorem}Let $G$ be an $(s,t)$-biregular graph of biorder $(m,n)$. Suppose that $G$ has a rotation $\rho$ giving a monofacial embedding and a unique-current assignment $\kappa$ from $Z_{2ms+1}$ satisfying {\em KCL}. Then there is an embedding of $K_{2ms+1}$ on an orientable surface such that each edge lies on a face of size $s$ and a face of size $t$. Moreover, if each local rotation on $G$ is simple with respect to $\kappa$, then the faces of $K_{2ms+1}$ are simple cycles. \end{theorem}

\begin{proof} This is the standard construction of a derived embedding from a current graph \cite{GT,R}; a careful analysis is given in \cite{GA}. The vertex set of $K_{2ms+1}$ will be the elements of $\mathbb Z_{2ms+1}$. Let $e_1,\dots,e_{2ms}$ denote the directed edges traversed in the single face of the embedding of $G$. Define the local rotation at vertex $0 \in \mathbb Z_{2ms+1}$ as $(\kappa(e_1),\dots,\kappa(e_{2ms}))$. Develop this rotation in $\mathbb Z_{2ms+1}$ by defining the rotation at vertex $i \in \mathbb Z_{2ms+1}$ as $(\kappa(e_1)+i,\dots,\kappa(e_{2ms})+i)$.

We use the face-tracing algorithm to show that a vertex of degree $d$ in $G$ satisfying {\em KCL} corresponds to $2ms+1$ faces of size $d$ in the embedding of $K_{2ms+1}$. Since the graph is $(s,t)$-biregular, each edge of $K_{2ms+1}$ lies on faces of size $s$ and $t$. Likewise, if each local rotation is simple, then the corresponding faces of the embedding of $K_{2ms+1}$ are simple cycles. \end{proof}

The rotation $\rho$ on the current graph $G$ plays two independent roles in this construction: {\em i}) $\rho$ generates a monofacial embedding, and {\em ii}) each local rotation $\rho_v$ is simple with respect to the current assignment $\kappa$. Any rotation on a vertex of degree $d \le 5$ is simple.

There is a quick way to find the faces of the derived embedding arising from an embedded current graph. Let $v$ be a vertex of the current graph $G$ with degree $k$. The rotation $\rho$ giving the monofacial embedding determines a local rotation $\rho_v = (e_1,\dots,e_k)$ on $D_v$. Consider $s_i = \sum_{j=1}^i \kappa(e_j)$. By {\em KCL} $s_k \equiv 0 \pmod {2ms+1}$. The faces of the embedding of $K_{2ms+1}$ are precisely the cyclic shifts of $(0,s_1,\dots,s_{k-1})$. Describing the faces is enough to determine the embedding. A full proof of Theorem \ref{big_theorem} shows that the monofacial condition guarantees that these faces meet in a cyclic manner at each vertex, i.e., we have a surface without pinch points. 

%%%%%%%%%%%%%%%%%%%%%%%%%%%%%%%%%%%%%%%%%%%%%%%%%%%%%%%%
\section {Relating Heffter arrays and current graphs}\label{relation}
%%%%%%%%%%%%%%%%%%%%%%%%%%%%%%%%%%%%%%%%%%%%%%%%%%%%%%%%

We have described two seemingly different objects: Heffter arrays and current graphs on biregular graphs. We show they are closely related.

\begin{proposition}\label{array_to_current} A Heffter array $H = H(m,n;s,t)$ is equivalent to a unique-current assignment $\kappa$ on a $(s,t)$-biregular graph $G$ of biorder $(m,n)$. This graph is connected if and only if $H$ is not block diagonal. 

Two compatible orderings $\omega_r$ and $\omega_c$ on the row and column Heffter systems of $H$ are equivalent to a monofacial rotation $\rho$ on $G$. Moreover, if $\omega_r$ and $\omega_c$ are both simple, then $\rho$ is simple with respect to $\kappa$. \end{proposition}

\begin{proof} Given a Heffter array $H(m,n;s,t)$, form a bipartite graph $G$ whose vertex set is the rows of $H$ together with its columns. For each non-empty $a_{i,j}$ in $H$ add an edge in $G$ labeled with current $a_{i,j}$ directed from the $i^{th}$ row of $H$ to its $j^{th}$ column; as usual the reverse edge receives the negative current. Since $H$ has $s$ entries per row and $t$ per column, the resulting graph is $(s,t)$-biregular graph of biorder $(m,n)$. Each row and column of $H$ sums to 0, so $G$ satisfies {\em KCL}. The entries of $H$ form a half-set $L$, so $G$ has unique currents. If $G$ is disconnected, then the components give partition of rows and columns showing that $H$ is block diagonal. 

Relating the orderings $\omega_r$ and $\omega_c$ on $H$ to the rotation $\rho$ on $G$ is more difficult. We use the unique currents to describe $\rho$ not as a permutation of directed edges but rather as a permutation of their nonzero currents. We use $\tau(a)$ for $-a$ reflecting oppositely directed edges receive inverse group elements. Define $\rho : \mathbb Z_{2ms+1}\setminus \{0\} \rightarrow \mathbb Z_{2ms+1}\setminus \{0\}$ by:

$$\label{define_rho} \rho(a) = {\Big \{}
\begin{array}{c@{\quad\quad}l}
\omega_r(a) & a \in L \\
\tau \circ \omega_c \circ \tau(a) & a \notin L
\end{array}.$$ 

Note that if $a \in L$, then $(\rho \circ \tau)^2 (a) = \omega_r \circ \omega_c (a)$. Since $\omega_r$ is compatible with $\omega_c$, $(\rho \circ \tau)^2$ acts cyclically on $L$. The odd powers of $\rho \circ \tau$ act cyclically on $-L$. Hence $\rho$ acts cyclically on $\mathbb Z_{2ms+1}\setminus \{0\}$ and the embedding is monofacial as desired. 

The reverse of the construction above shows that a current assignment corresponds to a Heffter array, hence the equivalence.  \end{proof} 

We combine our results for the following: \medskip

\begin{proof} (of Theorem \ref{main}) Apply Theorem \ref{array_to_current} to a Heffter array to build a current graph. Theorem \ref{big_theorem} then gives the desired embedding. \end{proof}

We illustrate the relation with an example. Figure \ref{three_by_four} gives a Heffter array $H(3,4)$ with entries from $\mathbb Z_{25}$. We simply order the parts of the row system $\omega_r = (1,-2,-10,11)(-8,6,-3,5)(7,-4,-12,9)$ and simply order the column system $\omega_c = (1,-8,7)(-2,6,-4)(-10,-3,-12)(11,5,9)$. Together these give the rotation $\rho$ given in Equation \ref{rho}, agreeing with the projected rotation on the $K_{3,4}$ of Figure \ref{unsigned_k34}. In turn this gives the monofacial face $\rho \circ \tau$ of Equation \ref{rhotau}. 

This current graph gives an embedding of $K_{25}$ on an orientable surface with every face on the boundary of triangle and a quadrilateral, i.e., a biembedding of a 3-cycle system with a 4-cycle system. Euler's formula implies this surface is of genus 49.

%%%%%%%%%%%%%%%%%%%%%%%%%%%%%%%%%%%%%%%%%%%%%%%%%%%%%%%%
\section {Weak Heffter arrays and nonorientable \\ embeddings}\label{nonorientable}
%%%%%%%%%%%%%%%%%%%%%%%%%%%%%%%%%%%%%%%%%%%%%%%%%%%%%%%%

A variation of Heffter arrays corresponds to signed current graphs embedded on nonorientable surfaces. In turn, this can be used to construct nonorientable $\{s,t\}$-biembeddings of complete graphs. In this section we describe this relationship begining with embeddings of signed graphs.

%===============================
\subsection{Signed current graphs}
%===============================

A {\em signed graph} $G^\pm$ is a graph $G$ together with a signature $\sigma : E(G) \rightarrow \{ +,-\}$. The {\em signature of a cycle} in $G$ is the product of the signatures on its edges. A {\em local switch} at $v$ toggles the sign of each edge incident with $v$. Two signatures are {\em equivalent} if and only if they are related by a sequence of local switches. A signature is equivalent to the all-positive signature if and only if the negative edges form a co-cycle, i.e., if there is no odd-length negative cycle. In this case the signature is called {\em balanced}, otherwise it is {\em unbalanced}. 

We describe {\em signed-embeddings} of signed graphs. As before, for each vertex we give a cyclic permutation $\rho_v$ of $D_v$. We keep track of the {\em local sense} of the orientation, one of two states {\em anticlockwise} or {\em clockwise}. When the local sense is anticlockwise, a face boundary entering a vertex $v$ on a directed edge $e^+$ leaves along $\rho_v(e^-)$. When it is clockwise we exit $v$ along the directed edge $\rho^{-1}_v(e^-)$. When traversing a negative edge we reverse our local sense of orientation. The face boundary closes when we reach a directed edge previously traversed in the same local sense. This process of tracing the orbits of $\rho \circ \tau$ while keeping sense of the local sense of orientation is called the {\em modified face-tracing algorithm}. The resulting surface is orientable if and only if the signature is balanced. 

A {\em signed current assignment} is a function $\kappa : E(G^\pm) \rightarrow \mathbb Z_m$ such that $\kappa (e^-) = -\kappa (e^+)$ when $\sigma(e)$ is positive, and $\kappa (e^-)=\kappa(e^+)$ when $\sigma(e)$ is negative. In analogy with unsigned current assignments we frequently require the following: 

\begin{enumerate}
\item (Kirchoff's Current Law - {\em KCL}) For every vertex $v$, \ $\sum_{e \in D_v} \kappa (e) \equiv 0 \pmod m$,
\item (Signed-Unique Currents) A current $\kappa$ occurs on a unique directed edge unless that edge is signed negatively, where one of $\kappa$ or $-\kappa$ appears twice and the other not at all, and
\item (Monofacial) $G^\pm$ is signed-embedded on a surface with a single face (this property of $G^\pm$ is independent of $\kappa$).
\end{enumerate}

\begin{figure}[htb]
\begin{center} 
\includegraphics[width=.5\textwidth]{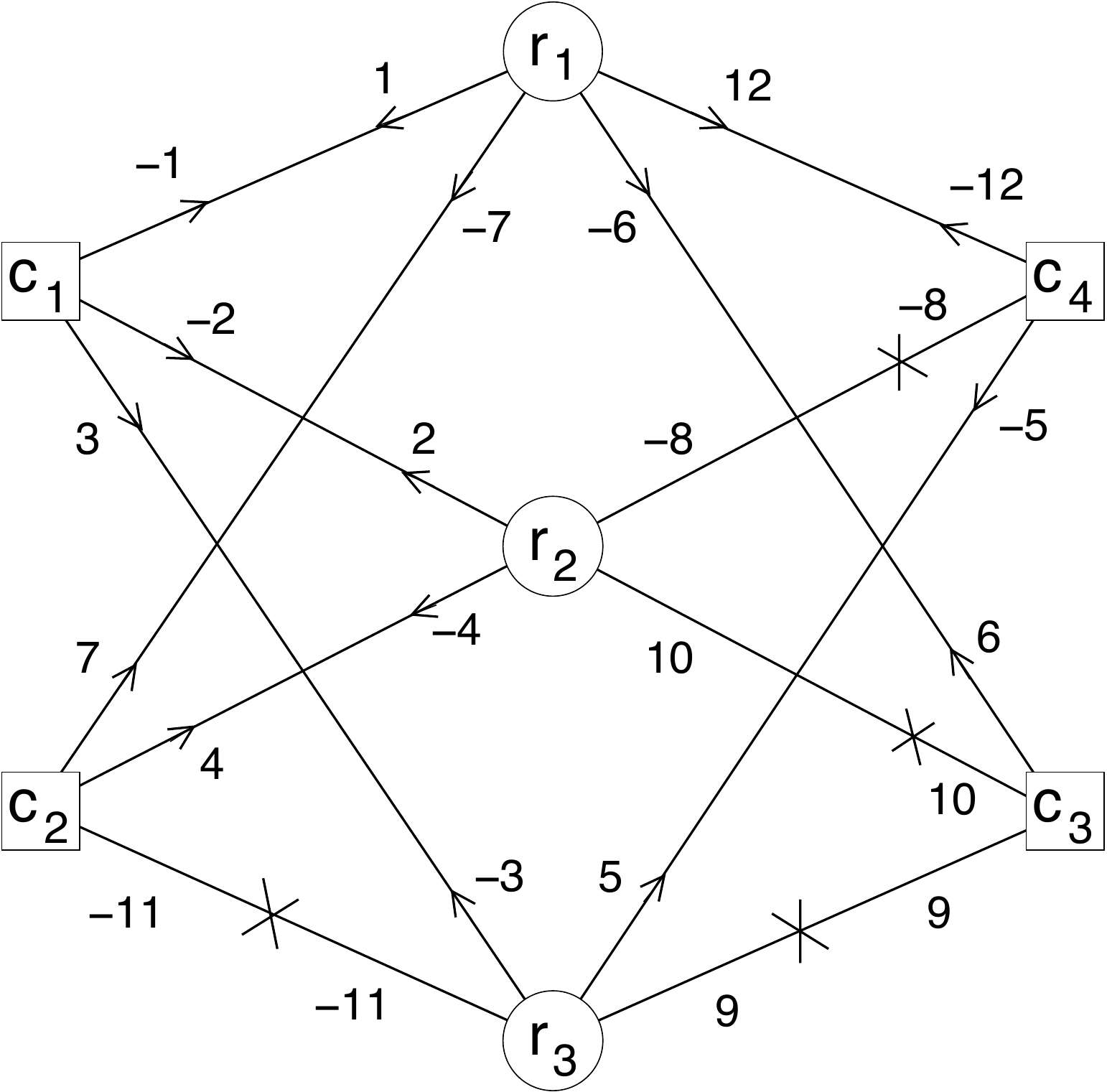} 
\end{center} 
\caption{\label{signed_k34} A signed current graph}
\end{figure}

For example, a signed current graph on $K_{3,4}$ is shown in Figure \ref{signed_k34}. Using the notation of Youngs \cite{Y} the negatively signed edges are indicated with an $\times$ in the middle; they receive currents -8,9,10,-11. The current assignment in the figure satisfies {\em KCL} and has signed-
s. A local rotation at each vertex is read off anticlockwise in this projection diagram. Using this rotation and signature the modified face-tracing algorithm gives the single face:
\begin{eqnarray}
\nonumber (1,3,-11,7,1,-2,-8,-5,-3,-2,-4,7,-6, \\
\nonumber 10,-4,-11,9,10,-8,-12,-6,9,5,-12).
\end{eqnarray}

The restriction of $\rho$ to $D_v$ gives a cyclic permutation of the currents $\kappa(e^+)$, $e^+ \in D(v)$. As before, define $\rho$ to be {\em simple at $v$} if the corresponding partial sums on $D_v$ are all distinct, and to be {\em simple} if $\rho$ is simple at each vertex $v$.

The following is analogous to Theorem \ref{big_theorem}.

\begin{theorem}\label{big_theorem_2} Let $G^\pm$ be an $(s,t)$-biregular unbalanced signed graph of order $(m,n)$. Suppose that $G^\pm$ has a monofacial nonorientable embedding and a signed unique-current assignment from $Z_{2ms+1}$ satisfying {\em KCL}. Then there is an embedding of $K_{2ms+1}$ on a nonorientable surface such that each edge lies on a face of size $s$ and a face of size $t$. Moreover, if each local rotation on $G$ is simple, then the faces of $K_{2ms+1}$ are simple cycles. \end{theorem}

\begin{proof} This is again a standard construction of a derived embedding from a current graph, see \cite{GT,R} for the full proof. The vertex set of $K_{2ms+1}$ is the elements of $\mathbb Z_{2ms+1}$. The monofacial signed embedding of the current graph is used to determine the rotation at a vertex of the derived graph. A lift of the signature on the current graph determines the signature on the derived complete graph.

Since the embedding underlying the current graph is unbalanced, it is in a nonorientable surface. Hence there is a negatively signed cycle. This in turn implies a negatively signed cycle in the derived embedding, i.e., it also is nonorientable. Finally, simple local rotations in the current graph correspond to simple cycles for faces in the derived graph as before. \end{proof}

Using the modified face-tracing algorithm there is a quick way to find the faces of the derived embedding arising from a embedded signed-current graph. Let $v$ be a vertex of the current graph $G^\pm$ with degree $k$. The rotation $\rho$ gives determines a local rotation $\rho_v = (e_1,\dots,e_k)$ on $D_v$. Consider $s_i = \sum_{j=1}^i \kappa(e_j)$. By {\em KCL} $s_k \equiv 0 \pmod {2ms+1}$. The faces of the embedding of $K_{2ms+1}$ are precisely the cyclic shifts of $(0,s_1,\dots,s_{k-1})$. Describing the faces is enough to determine the embedding. A full proof of Theorem \ref{big_theorem_2} shows that the monofacial condition guarantees that these faces again meet in a cyclic manner at each vertex. 

A key feature in the above construction is a monofacial embedding of a signed graph $G$. The following theorem addresses when such embeddings exist. Let $T$ be a positively signed spanning tree of $G$. A cotree component is {\em odd} if it has an odd number of edges and every negative edge is a bridge separating that component into two parts each having an odd number of edges.

\begin{theorem} A signed graph has a monofacial signed embedding if and only if there is a spanning tree with no odd components. \end{theorem}

To apply Theorem \ref{big_theorem_2} this monofacial embedding need also induce a simple ordering at each vertex. We have no general theory for this and it remains a project for future research. 

%===============================
\subsection{Weak Heffter arrays}
%===============================

Two Heffter systems $D_1(2ms+1,s)$ and $D_2(2nt+1,t)$ with $ms = nt$ are {\em weakly sub-orthogonal} if the $i^{th}$ part of $D_1$ has at most one element $a_{i,j}$ such that either $a_{i,j}$ or $-a_{i,j}$ is in the $j^{th}$ part of $D_2$. Form a {\em weak Heffter array} $H(m,n;s,t)$ by placing $a_{i,j}$ in row $i$ column $j$. The upper sign on $\pm$ or $\mp$ is the {\em row sign} corresponding to its sign on $a_{i,j}$ in $D_1$, the lower sign is the {\em column sign} used in $D_2$. Using the row signs we get row sums 0 and the column signs give column sums 0. Figure \ref{weak} shows a weak Heffter array. 

\begin{figure}[htb]
$$ 
\begin{array}{|c|c|c|c|} \hline
1 & -7 & -6 & 12 \\ \hline
2 & -4 & \pm 10 & \mp 8 \\ \hline
-3 & \mp 11 & \pm 9 & 5 \\ \hline
\end{array}
$$
\caption{\label{weak} A weak Heffter array $H(3,4)$ over $\mathbb Z_{25}$}
\end{figure}

We relate a weak Heffter array $H(m,n;s,t)$ to a signed current assignment. Form a bipartite $G$ of order $(m,n)$ whose vertices are the rows and columns of $H$.  For each nonempty cell add an edge $e$ labeled with the row-signed current $a_{i,j}$ directed from the $i^{th}$ row of $H$ to its $j^{th}$ column. The same edge in the opposite direction is assigned $-a_{i,j}$ and signed positively unless the entry is signed $\pm$ or $\mp$, in which case it's assigned $a_{i,j}$ and is signed negatively. Figure \ref{signed_k34} shows the signed current graph corresponding to the weak Heffter array of Figure \ref{weak}; the negatively signed edges are marked with an $\times$.

%%%%%%%%%%%%%%%%%%%%%%%%%%%%%%%%%%%%%%%%%%%%%%
\section {Conclusion}\label{conclusion}
%%%%%%%%%%%%%%%%%%%%%%%%%%%%%%%%%%%%%%%%%%%%%%

We have introduced Heffter arrays and their relation with current graphs and with biembeddings. The following table summarizes these relations.

\bigskip
\begin{center}
\begin{tabular}{|l|l|l|} \hline
{\bf Heffter array} & {\bf current graph} & {\bf biembedding} \\ \hline \hline
group & current group & vertex set \\ \hline
rows and columns & bipartition of vertices & face-2-colorable\\ \hline
\# entries row/col & biregular graph & biregular face sizes\\ \hline
zero row/col sums & KCL & faces of size $s,t$\\ \hline
simple order & rotation at a vertex & faces are simple cycles\\ \hline
compatible orders & monofacial embedding & no pinch points at vertices\\ \hline
\end{tabular}
\end{center}
\bigskip

In our definition of a Heffter array $H(m,n;s,t)$ we required the row and column sums to be 0 modulo $2ms+1$. The following tighter requirement is useful in constructions. Let $L$ be a halfset of $\{\pm k \ | \ k = 1,\dots,ms\}$. An {\em integer Heffter array} $H(m,n;s,t)$ is an $m \times n$ array with entries from $L$ such that each row and each column sum to 0 over the integers.

\begin{lemma}\label{necessary_integer} If an integer Heffter array $H(m,n;s,t)$ exists, then $ms \equiv 0,3 \pmod{4}$. \end{lemma}

\begin{proof} For a row to sum to 0 there must be an even number of odds. Hence the set of $ms$ entries must have an even number of odds, implying the congruence. \end{proof}

A third condition is also helpful in constructions. An integer Heffter array $H(m,n;s,t)$ is {\em shiftable} if each row and column contain the same number of positive as negative entries.  The array of Figure \ref{unbalanced} is integer and shiftable. Given a shiftable $H(m,n;s,t)$ with entries $a_{i,j}$, define $b_{i,j} = a_{i,j} + k$ if $a_{i,j} > 0$, and $b_{i,j} = a_{i,j} -k$ otherwise. The matrix $H \pm k$ whose entries are $b_{i,j}$ still has row and column sums 0 over the integers. Define the {\em support} of a matrix $A = \{a_{i,j}\}$ as $support(A) = \{|a_{i,j}|\}$. If the support of the original matrix is $\{1,\dots,ms\}$, then the support of the new matrix is $\{1+k,\dots,ms+k\}$.  

\begin{lemma}\label{necessary_shiftable} If a shiftable Heffter array $H(m,n;s,t)$ exists, then $s$ and $t$ are even both even, at least 4, and $ms \equiv 0 \pmod{4}$.\end{lemma}

\begin{proof} Since the number of positive and negative entries are the same in each row and column, $s$ and $t$ are both even. By Lemma \ref{necessary_integer} $ms \equiv 0,3 \pmod{4}$, so $ms \equiv 0 \pmod{4}$. \end{proof}

A natural question is the following, which we believe to be true.

\begin{conjecture}\label{big_conjecture} There exists a Heffter array $H(m,n;s,t)$ for all $m,n,s,t$ with $s,t \ge 3$ and $ms = nt$. If the conditions of Lemma \ref{necessary_integer} are satisfied, then there is an integer Heffter array. If the conditions of Lemma \ref{necessary_shiftable} are satisfied, then there is a shiftable array. \end{conjecture}

This conjecture is extensive since there are four parameters related by a single equation $ms = nt$. It is natural to examine special cases, such as {\em i)} when $H$ has no empty cells ($n=t$ and $m=s$), {\em ii)} squares ($n = m$), or {\em iii)} fixing $s$ and $t$. 

In the case the array has no empty cells Conjecture \ref{big_conjecture} simplifies to the following.

\begin{conjecture}\label{full} There exist Heffter arrays $H(m,n)$ for all $m,n \ge 3$. Moreover, they are integer Heffter arrays when $mn \equiv 0,3 \pmod{4}$, and are shiftable when $m$ and $n$ are both even and at least 4.\end{conjecture} 

The author with Tom Boothby and Jeff Dinitz believe we have a proof of Conjecture \ref{full}. We are writing up the details for publication in \cite{AB} and have a computer program for their construction. Another subsequent paper \cite{ADDY} examines integer square Heffter arrays including a complete charaterization of shiftable arrays.

In our definition of a Heffter array we require at most one appearance of elements in $\{x,-x\}$. What if we allowed two appearances either the same or different, or a multiset of $\lambda$ such elements. Using the difference-set construction this gives rise to $\lambda$-fold cycle systems. Let $K_n^\lambda$ be the complete multigraph on $n$ vertices where every pair of edges is jointed by $\lambda$ edges in parallel. The analogue of a Heffter square with higher $\lambda$ gives an embedding of $K_n^\lambda$ into a surface that is face 2-colorable, each color class being a $\lambda$-fold cycle system. For example, \cite{DGLM} uses current graphs to constuct 2-fold embeddings with all faces triangles. 

Recall that we need a simple ordering on the rows and one on the columns to ensure the resulting face boundaries are simple cycles. Moreover these two orderings should be compatable. These seem easy to find in practice, but hard to prove their existence in general. Alspach conjectures:

\begin{conjecture}\label{ordering} Every $A \subset \mathbb {Z}_n \setminus \{0\}$ has a simple ordering, i.e., $A$ can be ordered so that the partial sums are all distinct.  \end{conjecture}

The author together with Jeff Dinitz and Doug Stinson have made some progress \cite{ADS}, including verifying Conjecture \ref{ordering} for $n \le 25$. Bode and Harborth showed it was true for $|A| = n-1$. The general conjecture remains open.  

%%%%%%%%%%%%%%%%%%%%%%%%%%%%%%%%%%%%%%%%%%%%%%
%%%%%%%%%%%%% ACKNOWLEDGEMENTS %%%%%%%%%%%%%%%%
%%%%%%%%%%%%%%%%%%%%%%%%%%%%%%%%%%%%%%%%%%%%%%

\bigskip \bigskip
\noindent {\bf Acknowledgements:} The author thanks Tom Boothby, Melanie Brown, Jeff Dinitz, Diane Donovan, Mike Grannell, Terry Griggs, Thomas McCourt, Doug Stinson, Greg Warrington, Sule Yazici, and others for helpful discussions.

%%%%%%%%%%%%%%%%%%%%%%%%%%%%%%%%%%%%%%%%%%%%%%
%%%%%%%%%%%%%%% THE REFERENCES %%%%%%%%%%%%%%%%%%
%%%%%%%%%%%%%%%%%%%%%%%%%%%%%%%%%%%%%%%%%%%%%%

%%%%%%%%%%%%%%%%%%%%%%%%%%%%%%%%%%%%%%%%%%%%%%
%%%%%%%%%%%%%%%%%%%%%%%%%%%%%%%%%%%%%%%%%%%%%%
%%%%%%%%%%%%%%%%%%%%%%%%%%%%%%%%%%%%%%%%%%%%%%

\end{document}